\documentclass[11pt]{amsart} 
\usepackage{amsmath,amsfonts,amssymb,amsthm,amscd}
\usepackage{stmaryrd}

\usepackage{esint}

\usepackage{hyperref}
\usepackage{dsfont}

\usepackage{color}

\newtheorem{theorem}{Theorem}[section]
\newtheorem{thm}[theorem]{Theorem}
\newtheorem{prop}[theorem]{Proposition}
\newtheorem{lem}[theorem]{Lemma}

\newtheorem{cor}[theorem]{Corollary}

\theoremstyle{definition}

\theoremstyle{remark}

\newcommand{\N}{\mathbb{N}}
\newcommand{\Z}{\mathbb{Z}}
\newcommand{\Q}{\mathbb{Q}}
\newcommand{\R}{\mathbb{R}}

\newcommand{\gp}{{\mathbb  F}_{p}}             
\newcommand{\pea}{{\mathbb  Z}_{p}}           
\newcommand{\ffpea}{{\mathbb  Q}_{p}}   	
\newcommand{\zhat}{\widehat{\mathbb  Z}}          

\newcommand{\ff}{\varphi}                   
\newcommand{\al}{\alpha}                    
\newcommand{\be}{\beta}                     

\newcommand{\sub}{\subseteq}    
\newcommand{\st}{\;|\;}

\renewcommand{\o}{\mathcal{O}}
\renewcommand{\P}{\mathbb{P}}
\renewcommand{\d}{\mathbb{D}}

\newcommand{\ad}{\mathbb{A}}
\newcommand{\adq}{\mathbb{A}^{\rm{f}}_{\Q}}
\DeclareMathOperator{\Min}{Min}
\DeclareMathOperator{\Ker}{Ker}
\DeclareMathOperator{\Max}{Max}
\DeclareMathOperator{\spec}{Spec}

\newcommand{\tv}[1]{\llbracket {#1} \rrbracket} 
\newcommand{\fpp}{{\mathbb{F}_{2}}^{I}}
\newcommand{\cp}{\mathcal{P}}

\global\long\def\fp{\mathfrak{p}}
\global\long\def\fq{\mathfrak{q}}

\global\long\def\fm{\mathfrak{m}}
\global\long\def\fa{\mathfrak{a}}
\global\long\def\fb{\mathfrak{b}}

\global\long\def\id{\mathrm{id}}

\title[$\zhat$ and the ring of finite Ad\`eles over the $\Q$]{ Some model-theoretic perspectives on the structure sheaves  of $\zhat$ and the ring of finite ad\`eles over $\Q$}
\author{Paola D'Aquino}
\address{Dipartimento di  Matematica e Fisica, Universit\`a della Campania L.Vanvitelli, 
81100 Caserta, Italy. 
E-mail: paola.daquino@unicampania.it}
\author{Angus J. Macintyre}
\address{School of Mathematics, University of Edinburgh, EH9 3FD Edinburgh, U.K. E-mail: a.macintyre@qmul.ac.uk}
\author{Margarita Otero}
\address{Departamento de Matem\'aticas, Universidad Aut\'onoma de Madrid, 28049 Madrid, Spain. 
E-mail: margarita.otero@uam.es}
\thanks{The first author is partially supported by Programma Valere 2019 of Universit\`a della Campania "L. Vanvitelli", the third author is partially supported by  Spanish MTM2017-82105 and GR15/17}
\date{16 February 2020}
\begin{document}
\begin{abstract} We use the classical Ax-Kochen-Ershov  analysis of the model theory of  Henselian fields to bring out some model-theoretical aspects of the structure sheaf of  the spectrum of $\zhat$ and the ring of finite ad\`eles over $\Q$. We show that various structures associated to a prime ideal, such as quotients and localizations, are well understood model-theoretically, and they  are closely connected to   ultrafilters on the set of standard primes.

\end{abstract}
\maketitle

\section{Notation and Basic Notions}\label{basic}

We will use the following notation: 

$\mathbb{P}$ = the set of prime numbers; 

$\pea$ = the ring of $p$-adic integers;   

$\ffpea$ = the field of $p$-adic numbers;

$\mu_p$ = maximal ideal of $\pea$, and 
\[
\zhat=\prod_{p\in\mathbb{P}}\pea=\{f\colon\mathbb{P}\to\bigcup_{p\in\mathbb{P}}\pea\;|\;f(p)\in\pea\mbox{ for all }p\in\mathbb{P}\}.
\]
The ring $\zhat$ is a subring of $\prod_{p\in\mathbb{P}}\ffpea$, and the ring of finite ad\`eles  is the intermediate ring $$\adq=\{g\in\prod_{p\in\mathbb{P}}\ffpea|\;g(p)\in\pea\mbox{ for  cofinitely many  }p\in\mathbb{P}\}.$$
Note that $\adq$ is the localization of $\zhat$ at the multiplicative set of the positive integers $\Z^+$  (diagonally embedded in $\zhat$).

Note that $\Q$ is embedded in $\adq$ via the diagonal map and we
have $\adq\cong\zhat_{\Z^{*}}\cong\zhat\otimes\Q$. 

We consider the following topology on $\adq$. First of all, we recall
that, for each $p\in\P$, we have the $p$-adic norm on $\Q_{p}$
and both $\pea$ and $\Q_{p}$ are complete for the corresponding
metric, and the induced topology makes $\pea$ compact and $\Q_{p}$
locally compact (see, \emph{e.g.},\,\cite{gouvea}). For every
finite subset $S\subseteq\P$, let $\mathbb{A}(S)=\prod_{p\not\in S}\pea\times\prod_{p\in S}\Q_{p}$
endowed with the product topology and we define, for $U\subseteq\adq$,
$U$ open in $\adq$ if $U\cap\mathbb{A}(S)$ is an open subset of
$\mathbb{A}(S)$ for every finite $S\subseteq\P$. Note that $\zhat$
in open in $\adq$ and its induced topology coincides with the product
topology, hence $\zhat$ is a compact subring of $\adq$ and  so $\adq$
is locally compact (see  \cite{cassels}).

\

We use \cite[Chapter 6]{bosch} for definitions and formalism around $\spec$ and the structure sheaf. Let $R$  be a commutative ring with 1. $\Max R$  and $\Min R$ are respectively the set of maximal and minimal elements of $\spec R$  (under $\sub$).

$\spec(R)$ is the space of all prime ideals $\fp$
of $R$ with basic open sets $D(f)$, $f\in R$, where $D(f)=\{\fp\colon f\not\in\fp\}.$ $D(f)$ does not determine $f$. To $D(f)$ the structure sheaf $\mathcal{O}_R$ assigns $R_f$, the localization of $R$ at the powers of $f$. The issue of monotonicity is resolved by considering $$S(f)=\{g\in R\colon D(f)\sub D(g)\}$$ and showing that $S(f)$ is a multiplicative system  with $R_{S(f)}$  canonically isomorphic to $R_f$. The sheaf {on the category of basic opens} is extended to the structure sheaf of $\overline{\mathcal{O}}_R$ on $\spec(R)$, and we have the associated notion  of stalk.

In general, for  a ring $R$ contained in  a product of rings   $\prod_{i\in I}R_i$, an element $f\in R$ and  a (first-order) property $\ff$ of elements of a ring
we will extensively use the following classical  notation
$$
\tv{\ff(f)}=\{i\in I\st\ff(f(i))\text{ is true in }R_i\}.
$$

\section{Products of fields}

We first consider the simplest case $\fpp$, where $I$ is a set.
The power set of $I$ as a Boolean ring  is to be construed as  $\fpp$, and of course we identify $X\sub I$ with its characteristic function $e_X$  (where $e_X(i)=1$ if $i\in X$ and $e_X(i)=0$ if $i\not\in X$). Prime ideals in $\fpp$ are maximal. $\fpp$ is a von Neumann regular ring since all elements are idempotent.
 Ideals are exactly of the form $\{1-e_X\colon X\in \mathcal{F}\}$ with $\mathcal{F}$ a filter on $ I$,  and maximal ideals correspond to maximal filters, \emph{i.e.}, ultrafilters.

$\be I$, the Stone-$\check{\mathrm{C}}$ech  compactification of $ I$, is the set of ultrafilters on $ I$, topologized by taking as basic open sets the 
$\{\d\in\be I\colon Y\not\in\d\}$, for $Y$  a subset of $ I$. $\be I$ is compact Hausdorff, and is homeomorphic to $\spec(\fpp)$ by above description of the ideals of $\fpp$.

Open sets in $\be I$ are (by definition) of the form $\{\d\colon a\not\in\d\text{ for some }a\in \mathcal{A}\}$ where $ \mathcal{A}$ is a subset of the power set of  $ I$. This is the same as $\{\d\colon  I\setminus a\in\d\text{ for some }a\in \mathcal{A}\}$  and so  $\{\d\colon \mathcal{F}\sub\d\}$ where $ \mathcal{F}$ is the filter generated by all $ I\setminus a$, $a\in\mathcal{A}$.  So open sets in $\be I$ corresponds to filters \cite{comfort}.

Let $R=\fpp$. Consider a basic open set $D(f)$ in $\spec(R)$, for some $f\in R$. $\o(D(f))=R_f$. The multiplicative system is $\{f\}$ since $f$ is idempotent, and the kernel of the localization map $R\to R_f$ is  $\{g\in R\colon gf=0\}$, \emph{i.e.} the ideal  generated by $1-f$. Since $f$ maps to 1, \begin{equation}\label{F2}
R_f\cong R/(1-f)\cong {\mathbb{F}_2}^{\tv{f\not=0}}.
\end{equation}
Now we calculate 
 $$\overline{\o}(U)= \underset{\underset{D(f)\sub U}{f\in R,}}{\varprojlim} R_f,$$ for a general $U\sub\spec(R)$. 
  Because of $D(f)\cap D(g)=D(fg)$, for every $f,g\in R$, we have a directed system consisting of the $D(f)\sub U$. When $D(g)\sub D(f)$ localization gives a connecting morphism $$R_f\cong R_{S(f)} \longrightarrow R_{S(g)}\cong R_g,$$ and the preceding (see (\ref{F2})) discussion shows that $R_f \to R_g$ is the canonical  restriction ${\mathbb{F}_2}^{{\tv{f\not=0}}}\to {\mathbb{F}_2}^{{ \tv{g\not=0}}}$    (the condition ${\tv{g\not=0}}\sub{\tv{f\not=0}}$ is equivalent to $f\,|\,g$ since $f$ and $g$ are idempotents).
  
  Now, finally we connect to the discussion of open sets in {$\spec(\fpp)$}. Corresponding to $U$ we have the projective system of $R_f$, for $D(f)\sub U$ in exact ring-theoretic correspondence with the projective system  of maps $${\mathbb{F}_2}^{{\tv{f\not=0}}}\to {\mathbb{F}_2}^{{\tv{g\not=0}}},  \text{ for }f\,|\,g.$$

For each $U\sub \spec(\fpp)$, you can consider the filter $\mathcal{F}$ on $ I$ generated by the set $\{\tv{f\not=0}\colon D(f)\sub U\}$, and the limit of the projective system of $R_f$ is naturally isomorphic to the reduced product $\prod_\mathcal{F}{\mathbb{F}_2}$. 

 The correspondence $U\longleftrightarrow \mathcal{F}$ is natural, and functorially  $$\textstyle\overline{\o}(U)\cong\prod_\mathcal{F}{\mathbb{F}_2}.$$
The stalk at a point $\fp$ is, by essentially the same argument, identify as $\prod_\d{\mathbb{F}_2}$ where $\d$ is the ultrafilter of all $\tv{f=0}$, $f\in\fp$. Then of course the stalk is ${\mathbb{F}_2}$.

\

In the remainder of the paper we replace ${\mathbb{F}_2}^ I$ by subrings of  $\prod_{i\in I}R_i$, for different classes of rings $R_i$. 

\

 We now work with $R=\prod_{i\in I}K_i$, where the $K_i$ are fields, for instance $R=\prod_{p\in\P}\ffpea$. We subsume the case $\fpp$ discussed above. 
 $R$ is von Neumann regular, \emph{i.e.} every principal ideal is generated by an idempotent. The idempotents form a Boolean algebra (and thus a Boolean ring, though not a subring of $R$).

 Prime  ideals of $R$ are maximal. Ideals correspond exactly to {filters} on $I$, in the sense that any ideal is {uniquely} of the form $\{e_{I\setminus X}: X\in \mathcal{F}\}$, where $\mathcal{F}$ is a filter. Now not every element is an idempotent (only a unit times an idempotent). Note that maximal ideals correspond exactly to ultrafilters and principal ideals to principal filters. 

$\spec(R)$ is homeomorphic to $\be I$, the space of ultrafilters on $I$.
 {But of course the stalks and sections depend on the $K_i$}.

As for $\fpp$, $\overline{\o}(U)$ is a reduced product $\prod_\mathcal{F}K_i$, with $\mathcal{F}$ a filter on $I$ (and the correspondence is exactly as for $\fpp$). Again the stalks correspond exactly to ultraproducts $\prod_\d K_i$, $\d$ a ultrafilter. For $\fp\in\spec R$, both $R_\fp$ and $R/\fp$ are isomorphic to the ultraproduct $\prod_\d K_i$.

\section{Products of local domains and  valuation rings}\label{prodvaldom}
Now we pass from products of fields $R=\prod_{i\in I}K_i$ to 

\begin{itemize}

  \item[(a)] Product of local domains $R=\prod_{i\in I}R_i$,
  
\item[(b)] Product of valuation rings $R=\prod_{i\in I}V_i$.

\end{itemize}

The main example for us is $R=\prod_{p\in\P}\pea$ (\emph{i.e.} $\zhat$). In both cases 
$R$ is  no longer von Neumann regular in general. For example, the ideal generated by $f$ in $\prod_{p\in\P}\pea$, where $f(p)=p$, for $p\in\P$, is not generated by an idempotent. 

\
We do not reach decisive results about $\overline{\o}(U)$ now, but we can get precise information about quotients and localizations. 

\

a) Recall that a local ring $S$ is a commutative ring with a unique maximal ideal $\mu$. $\mu$ is first-order definable as the set of nonunits of $S$, the quotient $S/\mu$, the residue field of $S$, is interpretable in $S$.  So by the \L os theorem, the class of local rings is closed under ultraproducts, and both  $\mu$ and $S/\mu$ commute with ultraproducts:

If the $R_i$ are local rings, for $i\in I$, $\mu(R_i)=\mu_i$ are the corresponding maximal ideals, $k(R_i)=R_i/\mu_i$ the corresponding residue fields, and $\d$ an ultrafilter on $I$ then $\prod_{\d}R_i$ is a local ring whose maximal ideal is $\mu(\prod_{\d}R_i)=\prod_{\d}\mu_i$ and residue field $k(\prod_{\d}R_i)=\prod_{\d}R_i/\mu_i$.  Let 
$$\pi_{\d}:R\to\textstyle\prod_{\d}R_i:f\mapsto f_{\d}$$
be the canonical ring homomorphism onto the ultraproduct $\prod_{\d}R_i$.

Now we define two maps from the set of ultrafilters $\be I$ to $\spec(R)$.

 Let $\d\in\be I$, and define 
\[
\d_{*}:=\{f\in R:\tv{f=0}\in\d\}(= \{f\in R:\{i\in I:f(i)=0\}\in\d\}).
\]
and
\[
\d^{*}:=\{f\in R:\tv{f\in\mu}\in\d\}(=\{f\in R:\{i\in I:f(i)\in\mu_i\}\in\d\}).
\]
We see below that $\d_{*}$  and $\d^*$  are  prime ideals and belong to $\Min R$ and $\Max R$,  respectevely.

Conversely, starting with   any proper ideal $\fa $  of $ R$, we define $$\d_{\fa}:= \{\tv{f\in\mu}\colon f\in \fa\}.$$

\begin{lem}
\label{monotonicity}Let $f\in R$ and  $\fa $ a proper ideal of $R$. 

$(1)$    For any $f\in R$, $\tv{f\in\mu}=\emptyset$ iff $f$ is a unit of $R$.

$(2)$  For any $f\in R$,  $1-e_{\tv{f\in\mu}}\in fR$. 

$(3)$  $\d_{\fa}$ is a  proper filter on $I$. 

$(4)$  Let $\fb$ be a proper ideal of $R$. If $\fa\sub \fb$ then $\d_{\fa}\sub\d_{\fb}$. 

$(5)$ $\d_{\fp}$ is an ultrafilter, for any $\fp\in\spec R$. 
\end{lem}
\begin{proof}

(1), (2) and (4) are clear. For (3), firstly, note that by (1) $\emptyset\not\in\d_{\fa}$ for $\fa$ is
proper.
Now, let $X=\tv{f\in\mu}$ and $Y=\tv{g\in\mu}$ be in $\d_{\fa}$, so that $f,g\in \fa$.
Since 
$X\cap Y=\tv{1=e_{X}e_{Y}}=\tv{1-e_{X}e_{Y}\in\mu}, 
$ and $1-e_{X}e_{Y}=1-e_{X}+e_{X}(1-e_{Y})$,
we get  $1-e_{X}e_{Y}\in \fa$ since 
 both $1-e_{X}$ and $1-e_{Y}$ belong to
$\fa$ by (2).
Let $X\sub I$ and  $\tv{f\in\mu}\sub X$, for some  $f\in \fa$. Then,  let $g(i)=0$ if $i\in X$ and $g(i)=1/f(i)$ if $i\not\in X$ (and hence $f(i)$ is a unit in $R_i$)
so $1-e_X=fg$, and thus 
  $1-e_{X}\in \fa$,
and this implies $X=\tv{1-e_{X}\in\mu}\in\d_{\fa}$.

(5) If $\fp$ is a prime ideal of $R$ then, for every $X\sub I$, $(1-e_{X})e_{X}=0\in \fp$
implies either $1-e_{X}\in \fp$ or $1-e_{I \setminus X}=e_{X}\in \fp$.
Hence, either $X=\tv{1-e_{X}\in\mu}\in\d_{\fp}$ or $I\setminus X=\tv{e_{X}\in\mu}\in\d_{\fp}$.
\end{proof}
Note that $\d_{*}=Ker\,\pi_{\d}$ is a prime ideal of $R$ and 
$R/\d_{*}\cong\prod_{\d}R_i.$
On the other hand, $\d^{*}=\pi_{\d}^{-1}(\prod_{\d}\mu_i)$ is a maximal ideal of $R$.
Indeed, the residue maps $res_{i}\colon R_i\to R_i/\mu_i$
induce a map $res_{\d}:\prod_{\d}R_i\to\prod_{\d}R_i/\mu_i$. The map
$res_{\d}\circ\pi_{\d}$ is a homomorphism onto the ultraproduct $\prod_{\d}R_i/\mu_i$
which is a field, and $\d^{*}=\{f\in R:\tv{f\in\mu}\in\d\}=\pi_{\d}^{-1}(\prod_{\d}\mu_i)=Ker(res_{\d}\circ\pi_{\d})$.

\

\begin{thm}\label{d*}
$(1)$  Let $\d$ be an ultrafilter on $I$.

\hspace{1cm} $(1.\rm{a})$  For any ideal $\fa$ of $R$, $(\d_{\fa})_{*}\sub \fa\sub(\d_{\fa})^{*}$.

\hspace{1cm} $(1.\rm{b})$ $\d_{\d_{*}}=\d=\d_{\mathbb{D}^{*}}$. 

$(2)$  $\Min R=\{\d_{*}:\d\in\beta I\}$ and $\Max R=\{\d^{*}:\d\in\beta I\}$.

$(3)$  Let $\fm\in\Max R.$ Then, $R/\fm\cong\prod_{\d_{\fm}}R_i/\mu_i$.

\end{thm}
\begin{proof}
(1.a) Let   $f\in(\d_{\fa})_{*}$. Then $\tv{f=0}\in\d_{\fa}$. Hence there
is a $g\in \fa$ such that $\tv{f=0}=\tv{g\in\mu}$. Thus, 
 $f\in gR$, hence $f\in \fa$. 

Now, let $f\in \fa$ then $\tv{f\in\mu}\in\d_{\fa}$ by definition,
so $f\in(\d_{\fa})^{*}$.

(1.b) Let $X\in\d_{\d_{*}}$. Then $X=\tv{f\in\mu}$ for some $f\in\d_{*}$,
so that $\tv{f=0}\in\d$. Since $\tv{f=0}\sub\tv{f\in\mu}$ we have
that $X=\tv{f\in\mu}\in\d$. Therefore, $\d_{\d_{*}}=\d$ since both
are ultrafilters. Hence, the second equality by Lemma \ref{monotonicity}(5)
since $\d_{*}\subseteq\d^{*}$.

(2) We show that all the $\d_{*}$ are minimal primes. Let $\fp\in \spec R$.
If $\fp\sub\d_{*}$ then $\d_{\fp}\sub\d_{\d_{*}}$, hence $\d_{\fp}=\d$
by (1.b). Now by (1.a) we get $\d_{*}\sub \fp$. Conversely, if $\fp\in\Min R$,
then, $(\d_{\fp})_{*}\sub \fp$ by (1.a), hence $\fp$ being minimal we get $(\d_{\fp})_{*}=\fp$.
The other equality is proved analogously.

(3)  By (2)  and (1.b) $\fm=\d^*$, where $\d=\d_{\fm}$ since $\fm$ is maximal. Hence $R/\fm\cong\prod_{\d}R_i/\prod_{\d}\mu_i\cong \prod_{\d} R_i/\mu_i$.
\end{proof}
\begin{cor}
\label{c1d*}

$(1)$  Every prime ideal of $R$ is contained in a unique maximal ideal,
\emph{i.e.}, the ring $R$ is a  $\mathrm{pm}$-ring.

$(2)$ Every maximal ideal of $R$ contains a unique minimal prime
ideal. 

\end{cor}
\begin{proof}
(1) Let $\fp$ be a prime ideal of $R$.
Then $\fp\sub(\d_{\fp})^{*}$, the latter being maximal, and for any maximal
$\fq$ containing $\fp$ we must have $\d_{\fp}\sub\d_{\fq}$, hence being
ultrafilters they coincide. Therefore $\fq\sub(\d_{\fp})^{*}$ by Theorem
\ref{d*}\,1(a), and so $\fq=(\d_{\fp})^{*}$.

(2) Similarly, considering the minimal prime ideal associated to the
relevant ultrafilter.
\end{proof}

\

b) We now consider products $R=\prod_{i\in I}V_i$, where the $V_i$ are valuation rings. A valuation ring is a domain whose set of ideals is linearly ordered by inclusion. Note that this property is not obviously a  first order property. However, a more common equivalent definition is the one  that asks only that the principal ideals are linearly ordered by inclusion.   The latter  which is first order is the one we are going to use, and so the class of valuation rings is closed under ultraproducts.  We make fundamental use of that equivalence in what follows. See  \cite{Matsumura} for such basic properties of valuation rings.  Valuation rings are local domains so  the preceding analysis applies to products of valuation rings, but now we can get results not true in the local rings case  via the natural valuation of the valuation rings.  If $V$ is a valuation ring we consider the chain of principal ideals  $$\mathcal{C}:\quad(1)<\cdots< (a)<\cdots <(0)$$  for any $a$ nonzero and nonunit in $V$, where $(a)< (b)$ if $a$ divides $b$  and $b$ does not divide $a$. This linear order is interpretable uniformly in all valuation rings and commutes with ultraproducts.   We define $\infty$ as $(0)$. There is a totally defined operation   $\oplus$ on the linear order $\mathcal{C}$ given by $(a)\oplus(b)=(ab)$.  This is a commutative operation and  $(a)\oplus\infty=\infty,$ for all $a\in V$. Let $\Gamma= \mathcal{C}\setminus\{\infty\}$.  It is easy to see that $\Gamma$ is an ordered abelian cancellative  semigroup having $(1)$ as identity element. We refer to it as the value semigroup of $V$. $\Gamma$  is the nonnegative part of an essentially unique ordered Abelian  group $\widetilde{\Gamma}$.  Now we define a map $v:V\to\Gamma \cup \{\infty\}$, sending $a$ to $(a)$, which it is easily proved to be a valuation on $V$, and  it extends naturally to the fraction field of $V$ taking values in $\widetilde{\Gamma}\cup\{\infty\}$.  This definition is easily   shown to be equivalent  the more standard one where   ${\Gamma}$ is $V^*/\mathrm{U}$, and   $\mathrm{U}$ is the multiplicative group of units of $V$, and the order is given by $x \mathrm{U}\leq y\mathrm{U}$ if $x$ divides $y$ in $V$; and formally $\infty$ is added.

 Note that by the definitions we have given, valuation and value semigroup commute with ultraproducts, that is,   for any $\d\in \be I$, the value semigroup $\Gamma(\prod_\d V_i)$ is (isomorphic to)  $\prod_\d\Gamma(V_i)$, and  the valuation map $$v:\textstyle\prod_\d V_i \to \textstyle\prod_\d\Gamma(V_i) \cup\{\infty\}$$ is the ultraproduct of the valuation maps $v:V_i\to \Gamma(V_i) \cup\{\infty\}$
 
 Besides the results for local rings, in the case of valuation rings we further have the following.
 
 \begin{thm}
\label{idealslo} Let $\fp\in\spec R$. The
set of ideals of $R$ containing $\fp$ is linearly ordered by
inclusion. 
\end{thm}
\begin{proof}
By Theorem \ref{d*}(2) $\fp\supseteq\d_{*}$ for some ultrafilter $\d$ of
$I$. Then, ideals of $R$ containing $\fp$ correspond to ideals of $R/\d_*\cong \prod_\d V_i$ containing the image of $\fp$ in $R/\d_*$. Since  $\prod_\d V_i$ is a valuation domain its set of ideals is linearly ordered.
\end{proof}
In the case of $\zhat$ if $\d$ is a nonprincipal ultrafilter then  the ultraproduct $\prod_\d\Z_p$ is a  model of the axioms for Henselian rings  with value semigroup a  model of Presburger  and the residue field is a model of the theory of finite prime fields.

\section{The spectrum of $\prod_{i\in I} V_i$}
Let  $R=\prod_{i\in I}V_i$ with the $V_i$ local rings.
$\spec R$  is reduced, by definition, since $ R$ is reduced.
$\spec  R$ is not connected if $|I|>1$ since there are idempotents in $ R$
which are not 0 or 1 (if $X\sub I$, $X\not=\emptyset, I$ then $\spec\, R=V_{R}(e_{X})\cup V_{R}(1-e_{X})$).
By Corollary\,\ref{c1d*} every prime ideal of $ R$ is contained
in a unique maximal ideal, hence by \cite[Theorem 1.2]{MO71pmrings}
$\spec R$ is a normal space, $\Max R$ is Hausdorff and
the map associating to each $\fp\in \spec R$ the maximal ideal
$(\d_{\fp})^{*}$ is the unique retraction of $\spec R$ onto $\Max R$.

The basic closed sets in the Zariski topology of $\spec\,R$ will be denoted by $V_{R}(f):=\{\fp\in \spec\,R:f\in \fp\}$,
for $f\in R$, and the basic closed sets of the Stone topology on
$\be I$ will be denoted by $\langle X\rangle:=\{\d\in\be I:X\in\d\}$,
for $X\sub I$.

\begin{thm}
\label{cspec} The maps $\alpha_{*}:\beta I\to\Min R:\d\mapsto\d_{*}$
and $\alpha^{*}:\beta I\to\Max R:\d\mapsto\d^{*}$ are homeomorphisms.
In particular, $\Min R$, $\Max R$ are compact Hausdorff. 
\end{thm}
\proof
We first check that $\al_{*}$ is injective. Let $\d_{1}\not=\d_{2}$.
Take $X\subseteq I$, $X\in\d_{1}\setminus\d_{2}$ then $1-e_{X}\in(\d_{1})_{*}$
and $e_{X}\in(\d_{2})_{*}$, hence $1-e_{X}\not\in(\d_{2})_{*}$.
By Theorem \ref{d*}(2) the map $\alpha_{*}$ is also onto. To check
that $\alpha_{*}$ is continuos let $f\in R$ and $V:=V_{R}(f)\cap\Min R$
a basic closed subset of $\Min R$. Then $(\al_{*})^{-1}(V)=\left\{ \d\in\beta I:\d_{*}\in V\right\} =\left\{ \d\in\beta I:f\in\d_{*}\right\} =\left\{ \d\in\beta I:\tv{f=0}\in\d\right\} =\langle\tv{f=0}\rangle$
which is closed. Now we prove that $(\alpha_{*})^{-1}$ is continuous.
Let $\left\langle X\right\rangle $ be a basic closed for some $X\sub I$.
Then
$\al_{*}(\left\langle X\right\rangle )=\left\{ \d_{*}:X\in\d\right\} =\left\{ \d_{*}:\tv{1-e_{X}=0}\in\d\right\} =V_{R}(1-e_{X})\cap\Min R$,
last equality by Theorem \ref{d*}(2). 

The proof that $\alpha^{*}$ is homeomorphism is similar.
\endproof

\noindent
\emph{Picture of $\spec R$.} It is a union of maximal chains (under $\sub$), a typical chain being 
$$\left. \begin{array}{cc}
\text{maximal} & (\d_{\fp})^*\\
& \uparrow\\
&\fp\\
&\uparrow\\
\text{minimal} & (\d_{\fp})_*
\end{array}\right\}\text{linear order}
$$
The chains are indexed by $\be I$.
\begin{lem} For any $\fp\in \spec R$, $(\d_{\fp})_*\not=(\d_{\fp})^*$, provided that none of the $V_i$ is a field.
\end{lem}
\proof Just observe that $R/(\d_{\fp})_*\cong \prod_{\d_{\fp}}V_i$, a valuation ring and $R/(\d_{\fp})_*\cong \prod_{\d_{\fp}}k_i$. By our {assumption} that $V_i$ is not a field (\emph{i.e.}, $\Gamma_i\not=\{0\}$), $(\d_{\fp})_*$ is not maximal.
\endproof

In the rest of this section {we assume the $V_i$ are valuation  rings}. The prime ideals $\fp$ of $R$ containing the minimal prime $\d_*$ ($\d\in\be I$) correspond exactly to the primes $\cp $ of  $R/\d_*$, \emph{i.e.} to the primes $\cp $ of the ultraproduct $\prod_{\d}V_i$, and the latter is a valuation ring with  valuation $v$  to the ultraproduct $\Gamma:=\prod_\d\Gamma_i$, where $\Gamma_i=\Gamma(V_i)$. The $\cp $ correspond exactly to convex subsemigroups $\Delta$ of $\Gamma$. To a prime ideal $\cp $ one associates  the convex subsemigroup $$\Delta_\cp=\{ \gamma\in \Gamma\colon \forall f_\d\in\cp , \gamma<v(f_\d)\},$$ and to a convex subsemigroup $\Delta$ one associates the prime ideal  $$\cp_\Delta = \{f_\d\in\textstyle\prod_{\d}V_i\colon v(f_\d)>\Delta\}.$$  The correspondence $$\Delta\longmapsto\cp $$is order reversing, with $\Gamma\mapsto \{0\}$ and $\{0\}\mapsto\mu$ the maximal ideal of $\prod_{\d}V_i$, and clearly $\Delta_{\mathcal{P}_\Delta}=\Delta$ and $\mathcal{P}_{\Delta_\mathcal{P}}=\mathcal{P}.$

\
 
Now we address the question of the length of the chains of prime ideals (each chain corresponding to an ultrafilter on $I$). Firstly, we consider the case of principal ultrafilters.
 
\begin{prop}
\label{princ1} 

$(1)$  If $\fp\in\spec R$ is  a principal ideal of $R$ then  $\d_{\fp}$ is a principal ultrafilter generated by $\left\{ i\right\} $, for some $i\in I$.

$(2)$ If $\d\in\be I$ is the principal ultrafilter generated by $\{i\}$
then $$\d_{*}=(1-e_{\{i\}})\quad\text{and}\quad\d^*=\{f\in R\st f(i)\in \mu_i\}.$$

$(3)$ If the chain for $\fp(\in\spec R)$ has length $2$ then $\d_\fp$ is a principal ultrafilter.
\end{prop}
\begin{proof}
(1) If $\fp=(f)=\prod_{i\in I}(f(i))$ is a prime ideal of $R$, then
$\prod_{i\in I}V_i/(f)\cong \prod_{i\in I}(V_i/(f(i)))$ is a domain, so  must be isomorphic to $V_{i}/(f({i}))$ for some $i\in I$. Since $f(j)$  must be a unit in $V_j$, for each $j\in I$, $ j\not=i$ we have 
 that $\{i\}=\tv{f\in\mu}\in\{\tv{g\in\mu}:g\in(f)\}=\d_{(f)}$.

(2) Let $\d\in\be\P$ be a principal ultrafilter, generated by $\{i\}$, say, 
then $\d_{*}=\{f\in R:\tv{f=0}\ni i\}=\{f\in R:f(i)=0\}=(1-e_{\{i\}})$ and
$\d^{*}=\{f\in R:\tv{f\in\mu}\ni i\}=\{f\in R\st f(i)\in \mu_i\}$.

(3)  Suppose $\d_\fp$ is nonprincipal then the value semigroup of the  valuation ring  $\prod_{\d_\fp}V_i$ is the ultraproduct $\prod_{\d_\fp}\Gamma_i$. Let $\Delta$ be a proper  nontrivial initial segment (closed under addition). Hence the chain of length 3  of subsemigroups $\{0\}\subset\Delta\subset \prod_{\d_p}\Gamma_i$ 
 gives rise  to a chain of prime ideals $\mu\supset \fq\supset \{0\}$ in the ultraproduct $\prod_{\d_\fp}V_i\cong R/(\d_\fp)_*$ which in turn gives a prime ideal $\pi_{\d_\fp}^{-1}(\fq)\in\spec R$ strictly between $(\d_\fp)_*$ and $(\d_\fp)^*$.
\end{proof}
Note that  neither the converse of (1) nor of (3) is  true in general. It suffices to get one of the  valuation rings, $V_j$ say, with  maximal ideal, $\mu_j$, not principal,  with Krull dimension greater than 1, and $\d$ the principal ultrafilter generated by $\{j\}$. Thus,  if $\fp=\prod_{i\in I, j\not=i}V_i\times\mu_j$ then $\fp$ is nonprincipal   but $\d_\fp=\d$ is  generated by $\{j\}$, and the chain for $\fp$ has length $>2$.  In Theorem\,\ref{quotientsch}   we  give various examples of  rings as the mentioned $V_j$.

When $R=\zhat$ (or all the $V_i$ are discrete valuation rings) we can say more about the principal ultrafilter case:
\begin{cor}
\label{princ} 

$(1)$ Let $(f)\in\spec\zhat$. Then there is a $p\in\P$
such that either $\left(f\right)=\left(1-e_{p}\right)$ or $(f)=(1-(1-p)e_{p})$.
 Moreover, $\d_{(f)}$ is a principal ultrafilter (generated
by $\left\{ p\right\} $). 

$(2)$ If $\d\in\be\P$ is the principal ultrafilter generated by $\{p\}$
then $\d_{*}=(1-e_{p})$ and $\d^{*}=(1-(1-p)e_{p})$. Moreover, $\d_{*}$
and $\d^{*}$ are the only two prime ideals associated to $\d$ (i.e., the $\fp\in\spec\zhat$ with $\d_\fp=\d$). 

$(3)$  For any $\fp\in \spec\,\zhat,$  the chain for $\fp$ has length $2$ iff $\fp$ is principal iff $\d_{\fp}$
is principal.

$(4)$  For any $\fp\in \spec\,\zhat,$ $\d_{\fp}$ is nonprincipal if and only
if $\fp\cap\Z=\left\{ 0\right\} $ and $\fp\not=(1-e_{p})$, for any $p\in\P$.  If $\d$ is generated by $\{p\}$  the image of $p\in\Z$ in $\zhat$  generates $\d^*$. In particular,
if $\fp\in\Min\,\zhat$ then $\fp\cap\Z=\left\{ 0\right\} $.

\end{cor}
\begin{proof}
(1) If $(f)=\prod_{q\in\P}(f(q))$ is a prime ideal of $\zhat$, then
$\prod_{q\in\P}(\Z_{q}/(f(q)))$ is a domain, so there are no different
primes $q_{1}$ and $q_{2}$ with $\Z_{q_{i}}/(f(q_{i}))\not=0$.
Hence, there is a prime $p$ such that $f(q)\in U(\mathbb{Z}_{q})$
for all $q\not=p$ and $(f(p))$ prime in $\pea$, so $f(p)\in\mu_{p}$.
Thus, we have that $$\{p\}=\tv{f\in\mu}\in\{\tv{g\in\mu}:g\in(f)\}=\d_{(f)}.$$

(2) Let $\d\in\be\P$ be the principal ultrafilter generated by $\{p\}$,
then $$\d_{*}=\{f\in\zhat:\tv{f=0}\ni p\}=\{f\in\zhat:f(p)=0\}=(1-e_p)$$ and
$$\d^{*}=\{f\in\zhat:\tv{f\in\mu}\ni p\}=\{f\in\zhat:f(p)\in p\pea\}=(1-(1-p)e_{p}).$$
Take now $\fp\in \spec\,\zhat$ associated to $\d$ and suppose that
$\d_{*}\subset \fp\subseteq\d^{*}$. Since $\prod_{\d}\pea\cong\pea$ via
the map $f_{\d}\mapsto f(p)$, the composition $\pi_{\d}$ with the
latter is the map $\pi_{p}:\zhat\to\pea:f\mapsto f(p)$, and so $\d_{*}=\ker\pi_{p}$.
Hence, since $\d_{*}\subset \fp$, $\pi_{p}(\fp)$ is a nonzero prime
ideal in $\pea$, so $\pi_{p}(\fp)=p\pea$, hence $\fp=\left\{ f\in\zhat:f(p)\in p\pea\right\} =\d^{*}$. 

(3) From (1) and (2) and the proof of (3) in the previous proposition.

(4) If $\d_{\fp}$ is principal then $\fp=\d_{*}$ or $\fp=\d^{*}$, where
$\d=\d_{\fp}$, and  since $p$ is invertible in every $\Z_q$ for any prime $q\not=p$ we have that $\d^*=(1-(1-p)e_{p})=(pe_\P)$, then $p\in\d^{*}\cap\Z$.
On the other hand, if there is $n\in \fp\cap\Z$ then there is $p\in\P$
belonging to $\fp$, and hence $1-(1-p)e_{p}\in \fp$, so $\d_{\fp}$ is
principal.
\end{proof}
But what happens if $\d_\fp$ is nonprincipal? Here set theory intervenes, even when $I$ has cardinality $\aleph_0$
(as it does when we specialize later to $\zhat$).

We make some remarks concerning $\zhat$ at this point. $I=\P$ which is countable, and the common value semigroup is $\N(=\omega)$. Thus if $\d$ is an element of $\be\P$, $\d$ {nonprincipal}, the ultraproduct $\prod_\d\pea$ has cardinal $2^{\aleph_0}$ and is $\aleph_1$-saturated \cite{changkeisler}. The ordered semigroup $\prod_\d\omega$ has cardinal $2^{\aleph_0}$, is $\aleph_1$-saturated, and since its theory is independent of $\d$ so is its isomorphism type, under the continuum hypothesis  assumption. In particular the order type of the set of convex subsemigroups of $\prod_\d\omega$ is independent of $\d$. This, together with the correspondence between the prime ideals and convex subsemigroups above mentioned, allows as to prove:
\begin{lem} Suppose $\mathcal{CH}$. Then for $R=\zhat$ the length of the chain is independent  of $\d$ for nonprincipal $\d$.
\end{lem}

It seems interesting to characterize the order-type in this case. But we pass to a brief discussion of what happens when $\mathcal{CH}$ fails.

There is a very large literature (see e.g. \cite{blassmilden,canjarnotCH,canjarcof, mildenshelahcof}) on possible cofinalities of $\omega^\omega/\d$. In particular there is a model of ZFC +$2^{\aleph_0}=\aleph_3$ with two ultrafilters $\d_1$ and $\d_2$ on $\P$ so that  the ultrapower $\prod_{\d_i}\omega$ has cofinality $\aleph_i$ (i=1,2) \cite{canjarnotCH}.

Now we translate this into information about chains  determined by $\d_1$ and $\d_2$, {\em i.e.}, into information about the chains of convex subsemigroups of $\prod_{\d_1}\omega$ and $\prod_{\d_2}\omega$ respectively. We naturally identify $\omega^\P$ with $\omega^\omega$.

Note that the order type of the chains in $\spec\zhat$ (with $(\d_i)_*$ at the {bottom}) is the reverse of the ordering of convex subsemigroups (with $\prod_{\d_i}\omega$ at {top}). {The latter orders are complete}.
Let  $\mathcal{C}_i$ be the ordering of the convex subsemigroups of $\prod_{\d_i}\omega$. We turn Canjar's cofinality results into cofinality results for the $\mathcal{C}_i$.
\begin{thm} $\mathcal{C}_i$ has cofinality $\aleph_i$ $(i=1,2)$.
\end{thm}
\proof  For $\al\in \prod_{\d_i}\omega$, let $[\al]$ be the least convex subsemigroup containing $\al$, $[\al]=\{\be\colon\be\leq n\al\text{ for some }n\in\N\} $. We need only prove the result for $i=1$ (the argument transcribes easily to one for $i=2$). 

Let $\gamma$ be an ordinal and $\{\Gamma_\lambda\colon \lambda<\gamma\}$ be an increasing chain of proper convex subsemigroups, with $\bigcup_{\lambda<\gamma}\Gamma_\lambda=\prod_{\d_1}\omega$ (this being the last element of the chain of convex subsemigroups). Note that by $\aleph_1$-saturation of the ultrapowers $\prod_{\d_1}\omega\not=[\alpha]$ for any $\alpha\in\prod_{\d_1}\omega$. Thus we can select $\al_\lambda>\Gamma_\lambda$, so that $\alpha_\lambda$ are monotone increasing and cofinal in $\prod_{\d_1}\omega$. So $\gamma\geq$ cofinality of $\prod_{\d_1}\omega =\aleph_1$. So cofinality of $\mathcal{C}_1\geq\aleph_1$. 

On the other hand,  we can choose $\{\be_\delta\colon\delta<\aleph_1\}$ cofinal in $\prod_{\d_1}\omega$, and then $\{[\be_\delta]\colon\delta<\aleph_1\}$ is cofinal in $\mathcal{C}_1$. Thus the cofinality of $\mathcal{C}_1\leq \aleph_1$.
\endproof
\begin{cor}The chain of prime ideals strictly between $(\d_i)_*$ and $(\d_i)^*$ has coinitiality $\aleph_i$ $(i=1,2)$.
\end{cor}

\noindent
\emph{Note.} By using more general results of Canjar \cite{canjarnotCH}
one can get similar results involving both cofinality and coinitiality. We are content to provide basic examples, and assume that very hard set theory is needed to survey all possibilities.

\

Henceforward, serious set theory is not central in our work. In particular we will emphasize {elementary equivalence} over {isomorphism}.

\section{The quotients $\prod_{i\in I}V_i/\fp$}\label{R/p}\label{quotient}

Let $R=\prod_{i\in I}V_i$ with $V_i$ valuation rings, and $\fp\in\spec R$. 
 If $\fp$ is maximal, $R/\fp$ is naturally isomorphic to the ultraproduct  of the residue fields $\prod_\d k_i$ and if $\fp$ is minimal $R/\fp$ is naturally isomorphic to the ultraproduct $\prod_\d V_i$. If $\fp$ is principal, say $\fp=(f)$,  then $R/(f)$ is isomorphic to  $V_i/(f(i))$, for some $i\in I$. If $\d$ is principal, say generated by $j$,  and $\fp$  belongs to the chain of $\d$  then 
since $\prod_{i\in I}V_i/\d_*\cong \prod_\d V_i\cong V_j$, we have  $\prod_{i\in I}V_i/\fp\cong V_j/\fp_j$, where $\fp_j$ is the image of $\fp/\d_*$ in $V_j$.
  In the rest of this section we restrict only  to  $\fp$  in the chain associated to a nonprincipal $\d\in\be I$. 
  Moreover, we assume  the  $ V_i$ Henselian   and the ultraproduct of the residue fields $\prod_\d k_i$ of characteristic 0. Note that this is true for $\zhat$.  In this case we  show that in a given chain all quotients $R/\fp$, $\fp$ nonmaximal, are elementary equivalent, but {give} no idea how many isomorphism types there are.
If $\cp $ is the prime ideal  in $\prod_\d V_i$  corresponding to $\fp$ then $R/\fp\cong \prod_\d V_i/\cp .$ We first prove using Fornasiero's embedding theorem that $R/\fp$
is isomorphic to a subring of $\prod_\d V_i$, to do that we  introduce some notation. 

Let $V:=\prod_\d V_i$ and $K$ its fraction field,  $\widetilde\Gamma$ the group generated by $\Gamma:=\prod_\d\Gamma_i$,  $v: K^*\to\widetilde\Gamma$ the corresponding valuation and  $k=\prod_\d V_i/\mu_i$ the residue field. The $V_i$ are  Henselian and so is $K$.  The field $k$ has characteristic 0,  hence  we can lift $k$ and assume that $k\sub V\sub K$ (see, \emph{e.g.}, \cite[Lemma\,3.8]{fornasiero}). By \cite[Lemma\,3.10]{fornasiero} there is a map (1-\emph{good section}) $s:\widetilde\Gamma\to K^*$ satisfying $s(0)=1, vs=\id, s(-\gamma)=(s(\gamma))^{-1}$, such that  the corresponding co-cycle $$\mathfrak{f}: \widetilde\Gamma\times\widetilde\Gamma\to K^*: \mathfrak{f}(\gamma,\delta)=\frac{s(\gamma)s(\delta)}{s(\gamma+\delta)}$$
has its image in $k^*$ (\emph{i.e.}, $\mathfrak{f}$ is a \emph{factor set}). By \cite[Lemma\,5.1]{fornasiero} we can assume that $K$ is a truncation-closed valued subfield of  $k(( \widetilde\Gamma,  \mathfrak{f}))$, where the latter is the additive group  $k(( \widetilde\Gamma))$ endowed with multiplication given by $t^\gamma t^\delta=\mathfrak{f}(\gamma,\delta)t^{\gamma+\delta},$ and its standard valuation.

Consider the subsemigroup $\Delta$ corresponding to $\cp$ as above and $\widetilde\Delta$, the subgroup of $\widetilde\Gamma$ generated by it. Hence  with our new notation $\cp=\{x\in V\st v(x)>\Delta\}$.  Note that  $k(( \widetilde\Delta,  \mathfrak{f}))\sub k(( \widetilde\Gamma,  \mathfrak{f}))$, as truncation-closed  valued subfield. Let $W=k(( \widetilde\Delta,  \mathfrak{f}))\cap V.$ So $W$ is a subring  of $V$ and clearly \begin{equation}
\label{directsum}
V=W\oplus \mathcal P,
\end{equation}
 $V/\mathcal P\cong W$  so $W$ is \emph{lift} of  the domain $V/\mathcal P$ in $V$.

 How unique is $W$, as subring of $\prod_\d V_i$ mapping onto $\Delta$ (so disjoint from $\cp$)? $W$ is clearly maximal with this property. 
  
 So we can say of $R/\fp$ that is isomorphic to {a} maximal subring of $\prod_\d V_i$ mapping {onto} $\Delta$ under $v$.   We have not been able to say much more in general, even for $\zhat$, in an outrageous universe. 
 
 \
 
{Assuming} $R=\zhat$ and $\mathcal{CH}$ holds we can be quite explicit (though not canonical). We sketch the situation, maintaining the notations from above.
 
 Now $\prod_\d\pea$ is $\aleph_1$-saturated and of cardinality $\aleph_1$, and by Ax-Kochen-Ershov is isomorphic to the ring of power series (with well-ordered support $<\omega_1$), with coefficients in $k_\d=\prod_\d\gp$, and exponents in $\prod_\d\omega$. $\Delta$ is an initial segment ($\not=\{0\}$) of the latter, and $\cp$ is the set of power series supported above $\Delta$. As a complementary $k_\d$-subspace we can take the power series supported on $\Delta$, $k_\d[[t^\Delta]]_{<\omega_1}$. Note that $\Delta$ is the nonnegative part of a model  of {Presburger} arithmetic (a $\Z$-group). What are the possibilities for $\Delta$? Well, $\Delta$ can be $\omega$. But this is the only possibility for countable $\Delta$, for unless $\Delta=\omega$ $\aleph_1$-saturation gives the existence of an element $\gamma\in\prod_\d\omega$, $\gamma >\omega$ but $\gamma<\Delta\setminus\omega$, contradicting that $\Delta$ is an initial segment. So, unless $\Delta=\omega$, $\Delta$ is of cardinal $\aleph_1$.
 
Obviously $\omega$ is not $\aleph_1$-saturated, but among the other $\Delta$ some are $\aleph_1$-saturated and some are not. $\prod_\d\omega$ is $\aleph_1$-saturated, and of course has cofinality $\aleph_1$. But no $\Delta$ of cofinality $\aleph_0$ can be $\aleph_1$-saturated.

 \begin{lem} There are nonstandard $\Delta$ of cofinality $\aleph_0$.
 \end{lem}
 \proof Select (via $\aleph_1$-saturation of $\prod_\d\omega$) a sequence $$\gamma_0<\gamma_1<\cdots<\gamma_n<\cdots$$ so that $\gamma_{n+1}>m\gamma_n$ for all $m\in\omega$. Let $\Delta$ be the least convex subsemigroup containing the $\gamma_j$. The latter form a cofinal $\omega$-sequence in $\Delta$. 
 \endproof
 \begin{lem} All $\Delta$ are  the nonnegative part of  models of Presburger arithmetic.
 \end{lem}
 \begin{lem}All $\Delta$ of cofinality $\aleph_1$ are $\aleph_1$-saturated.
 \end{lem}
 \proof
 Let $p(x,A)$ be a Presburger 1-type over a countable subset $A$ of $\Delta$, where $\Delta$ has cofinality $\aleph_1$. We use the quantifier-elimination of Presburger, and by some routine manipulations we can assume that $p$ is given by congruence conditions modulo standard integers and by conditions $$A_1<x<A_2$$
 where $A_1$ and $A_2$ are countable (if $p$ forces some $x=a$, $p$ \emph{is} realized).
  If $A_1<x<A_2$ defines a finite nonempty set in $\Delta$ then the type is obviously realized if consistent.
  If $A_1<x<A_2$ defines an infinite  set in $\Delta$ then  it clearly meets every congruence class modulo every standard integer,
   and so $p$ is realized.
 If $A_1<x$ defines the empty set in $\Delta$, then $A_1$ is cofinal, contradicting the cofinality of $\Delta$. If $x<A_2$ defines the empty set in $\Delta$, then $A_2\supseteq\{0\}$, and the type is inconsistent.
 \endproof
 \begin{cor} Any two $\Delta$ of cofinality $\aleph_1$ are isomorphic.
 \end{cor}
\proof Any two such $\Delta$ are saturated models of nonegative Presburger  of cardinal $\aleph_1.$
\endproof
Now we consider nonstandard $\Delta$ of cofinality $\aleph_0$.

The ordering of all $\Delta$ is complete (\emph{i.e.} has arbitrary suprema and infima), $\{0\}$ and $\omega$ the first two elements, $\prod_\d\omega$ the last.

Suppose $\Delta$ has cofinality $\aleph_0$, $\Delta\not=\omega$. We analyze the archimedean classes of $\Delta$. 
 These are naturally linearly ordered. Suppose there is a top class, that of $\gamma$, say. Let $\Delta'$ be the convex subset of elements below the class of $\gamma$. $\Delta'$ is a convex subsemigroup, and by saturation of $\prod_\d\omega$ $\Delta'$ has cofinality $\aleph_1$,
(suppose there is $A\sub \Delta'$ countable and cofinal on $\Delta'$, then consider the type $A<x, nx<\gamma, n\in\N$, this is realized in $\prod_\d\omega$,  the realization $\delta$ must be in $\Delta$ since $\gamma\in\Delta$ and $\Delta$ is convex, but cannot be since  $\Delta'<\delta<[\gamma]$) 
 and so is saturated. $\Delta'$ is clearly pure in $\Delta$, 
  so by {\cite[Theorem 2.8]{Prest}} is a direct summand.  
  Let $\Delta=G\oplus\Delta'$, $G$ is archimedean so embeddable in $\R$. Using $\aleph_1$-saturation (against constants from $G$) one easily sees that $G\cong\R$. 
 
 Thus there is at most \emph{one} isomorphism type of $\Delta$ if there is a top class.
 
 If there is no top class there is a cofinal $\omega$-sequence of classes in $\Delta$, say of $$\gamma_0 < \gamma_1 < \cdots<\gamma_m <\cdots$$
 Now  it is easily seen (again by $\aleph_1$-saturation) that there is  such a sequence in which each $\gamma$ is $n$-divisible for all $n$ (and $\not=0$).
  But clearly (from quantifier elimination for Presburger) the partial $\omega$-type in $(w_0,w_1,\dots)$ saying each $w_j>0$, $w_{j+1}$ is bigger than each $\Q$-linear combination of the $w_l,\;l\leq j$, and each $w_j$ is $n$-divisible for all $n$, is \emph{complete}. 
    Thus by $\aleph_1$-saturation, any two realizations of the complete type are automorphic in $\prod_\d\omega$, and the convex $\Delta$ determined by the realizations are \emph{isomorphic}. Thus gives:
 \begin{lem} There are $\Delta\not=\omega$ of cofinality $\aleph_0$, with no largest archimedean class, and any two are isomorphic (as ordered semigroups).
 \end{lem}
 \proof Existence is clear by $\aleph_1$-saturation. Uniqueness is proved above.
 \endproof
 
 What can we say about the other case, $\Delta\cong\R\oplus\Delta'$, where $\Delta'$ has cofinality $\aleph_1$? Are there such $\Delta$? We have shown that there is at most one isomorphism type of $\Delta$ (and our discussion shows that it must be an isomorphism type different from that of the preceding lemma).
 
 There are such $\Delta$ in abundance. For let $\gamma$ represent \emph{any} nonstandard archimedean class,
   and let $\Delta'$ consist of all $\delta$ in smaller archimedean classes (clearly $\Delta'$ is a convex subsemigroup). Let $\Delta$ consist of all archimedean classes less than or equal to the class of $\gamma$. The class of $\gamma$ is the top class of $\Delta$, 
    and $\Delta\cong \R\oplus\Delta'$. 
 
 So we have established a sharp limitation on $\Delta$  under our hypothesis.
 \begin{thm}$\mathcal{(CH)}$ There are 5 isomorphism types of convex subsemigroups $\Delta$ of $\prod_\d\omega$, namely the types of 
 
 \emph{(i)} $\{0\}$;
 
 \emph{(ii)}  $\omega$;
 
 \emph{(iii)} $\Delta$ of cofinality $\aleph_1$;
 
 \emph{(iv)} $\Delta$  of cofinality $\aleph_0$, and no top archimedean class, and
 
 \emph{(v)} $\Delta$  nonstandard, of cofinality $\aleph_0$ and top archimedean class.
 \end{thm}
 This in turns gives a classification of the quotients by $\fp$ for $\fp$ associated to a nonprincipal $\d$. As above we assume $\mathcal{CH}$.
 \begin{thm}$\mathcal{(CH)}$\label{quotientsch} Let $\d\in\be\P$ be nonprincipal. Then the isomorphism types of $\zhat/\fp$ where $\fp$ is associated to $\d$ are exactly those of:
 
 \emph{(i)} $\prod_\d\gp\;(=k_\d)$;
 
 \emph{(ii)} $k_\d\tv{t}$;
 
 \emph{(iii)}  $k_\d\tv{t^\Delta}_{<\omega_1}$, where $\Delta$ is a convex subsemigroup of $\prod_\d\omega$ of cofinality $\aleph_1$;
 
 \emph{(iv)} $k_\d\tv{t^\Delta}_{<\omega_1}$, where $\Delta$ is a convex subsemigroup of $\prod_\d\omega$ of cofinality $\aleph_0$ and no top archimedean class, and
 
 \emph{(v)} $k_\d\tv{t^\Delta}_{<\omega_1}$, where $\Delta$ is a nonstandard convex subsemigroup of $\prod_\d\omega$ of cofinality $\aleph_0$ and has top archimedean class.
 \end{thm}
 \proof
 Done
  \endproof
 
 \noindent
\emph{Note.} (i) occurs only for $\fp$ maximal. (iii) occurs for $\fp$ minimal, but for many other primes.

 If we drop $\mathcal{CH}$,  the only thing we can ensure is that $\Delta$, the value semigroup  of $\zhat/\fp$, is elementary equivalent to $\N$, hence we get the following
 \begin{thm}\label{quotients} Suppose $\d$ nonprincipal and $\fp$ associated to $\d$. Then, 
 
   \emph{(i)} $\zhat/\fp\equiv k_\d$ if $\fp$ is maximal, and 
 
  \emph{(ii)} $\zhat/\fp\equiv k_\d\tv{t}$ if $\fp$ nonmaximal.

 \end{thm}
 \proof  Clear from Ax-Kochen-Ershov theorem since $\d$ is nonprincipal, and  $\zhat/\fp\cong \prod_\d\pea/\cp$, the latter being Henselian.
 \endproof
 
 \begin{cor}\label{invertible}  If $\d\in\be\P$ in nonprincipal and $\fp\spec\zhat$ is in the chain of $\d$ then  the integers are invertible in $\zhat/\fp$.
 \end{cor}
 
 \noindent
\emph{Remark.} We presume that in \emph{ZFC} alone the options for isomorphism types are mind-boggling.

\section{The localizations  $(\prod_{i\in I}V_i)_\fp$}\label{localization}
Our aim is  to understand the model theory of each ring of sections over open sets of $\spec (R)$, and the theories of the stalks. In Section 2 we discussed the case of $R=\prod_{i\in I} K_i$ where $K_i$ are fields.  We  now analyze the case $R=\prod_{i\in I} V_i$ where $V_i$ is a valuation ring, and in the case of  $\zhat$ we characterize up to elementary equivalence the local rings $\zhat_{\fp}$ for $\fp \in \spec (\zhat)$. 

Let $R=\prod_{i\in I}V_i$, with the $V_i$   valuation rings with fraction fields $K_i$, and  0 characteristic residue fields $k_i$. Let $\fp\in\spec R$ in the chain associated to a  $\d\in\be I$ and $V=\prod_\d V_i\cong \prod_{i\in I} V_i/ \d_*$, as above. We use the notation of previous sections, in particular $\cp=\fp/\d_*$.
\begin{lem} Consider the localization morphism $\tau:R\to R_\fp$. Then $\Ker\tau=\d_*$.
\end{lem}
\proof
Let $f\in\Ker\tau$. Then, there is a $g\in R\setminus\fp$ such
that $fg=0$. 
Since $g\not\in\fp$ we have $g\not\in\d_{*}$, and so $f\in\d_*$.
Conversely, let $f\in\d_{*}$
and hence  $X:=\tv{f=0}\in\d$. Then, $fe_{X}=0$ and $e_{X}\not\in\d_{*}$ and
so,  $e_{X}\not\in\d^{*}$, in particular
$e_{X}\not\in\fp.$ Hence $f$ is 0 in $R_{\fp}$.
\endproof
Thus the image of $R$ under $\tau$ is (naturally) isomorphic to  the valuation domain $V$. $R_\fp$ is generated over $V$ by adjoining all $1/f$ where $f\not\in\fp$. In particular $R_{\fp}$ is a subring of  $K:=\prod_\d K_i$, the  field of fractions of $V$. The field $K$ is a Henselian valued field, with  valuation ring $V$, valuation map $v: K \to\widetilde{\Gamma}\cup\{\infty\}$, and residue field $k$.

\begin{cor}\label{clocalization}  There is a natural isomorphism $R_\fp\cong V_\cp$. In particular
$R_{\d_*}\cong \prod_\d K_i$ and $R_{\d^*}\cong \prod_\d V_i$.
\end{cor}
\proof  Both rings contain a copy of $V$, and since $\d_*\sub \fp$ we have that $g\not\in \fp$ if and only if $g_{\d}\not\in\cp$.
\endproof

 \begin{cor}\label{localizationhenselian} Assume all the $V_i$ are Henselian.  Then $R_\fp$ is Henselian.
 \end{cor}
\proof Consider the  valuation 
$v_\cp\colon K\to\widetilde{\Gamma}/\widetilde{\Delta}\cup\{\infty\},$ corresponding to the coarsening 
$V_\cp\supseteq V $, and the corresponding valuation $\overline{v_\cp}:V_\cp/\cp\to\widetilde{\Delta}\cup\{\infty\}$ on the residue class field of $v_\cp$. Then both $V_\cp$ and $V/\cp$ are Henselian \cite[Corollary 4.1.4]{englerprestel}. \endproof

\begin{thm}\label{pain}  Let $\d$ be a nonprincipal ultrafilter on $\P$. Then  all the  local rings $\zhat_\fp$,   for $\d_*\subset\fp$ are elementary equivalent.
\end{thm}
\proof Let $\fp$ be a prime associated to $\d$ and $\fp\not=\d_*$. Let $\cp$ be its image in $\prod_\d\Z_p$.  Let  $K=\prod_\d\Q_p$ and $V=\prod_\d\Z_p$. Since the  valued field $(K, V)$ is Henselian, so  $(K, V_\cp)$ and $(V_\cp/\cp V_\cp, V/\cp V_\cp)$ are (see \cite[Corollary 4.1.4]{englerprestel}). Its value groups are  respectively $\widetilde\Gamma\cong \prod_\d\Z$, $\widetilde\Gamma/\widetilde{\Delta_\cp}$ and $\widetilde{\Delta_\cp}$.   The  group $\widetilde\Gamma/\widetilde{\Delta_\cp}$ is an ordered divisible abelian group hence as $\fp$ varies, as above,  they are all elementary equivalent.   On the other hand, since $$(V/\cp V_\cp)/(\mu/\cp V_\cp)\cong V/\mu\cong \textstyle\prod_\d\gp,$$  hence, as $\fp$ as above, all the $V/\cp V_\cp$ have the same residue field. Therefore all the $V_\cp/\cp V_\cp$ are elementary equivalent. Finally since
 $\widetilde{\Delta_\cp}$ is a model of Presburger,  so as  $\fp$ varies, as above,   they are all elementary equivalent.  We can conclude that all the $V_\cp$ are elementary equivalent and so  the $\zhat_\fp$, as $\fp$ varies as above, are.
\endproof

\section{The ring of finite ad\`eles of $\Q$}

In what follows we will denote the ring of finite ad\`eles of $\Q\,$ $\adq$ by $\ad$, \emph{i.e.}
\begin{center}
$\ad:=\{ f\in\prod_{p\in\mathbb{P}}\Q_{p}:f(p)\in\pea\mbox{ for all but finitely many }p\in\P\}.$
\end{center}

The analysis  of $\spec\ad$ is easily obtained now from that of $\zhat$ via the isomorphism $\ad\cong\zhat_{\Z^{*}}$.  So we have a   bijection 
$$\spec\ad\rightarrow\{\fp\in\spec\zhat\st \fp\cap \zhat^*=\emptyset\}$$ mapping $\fq$ to $\fq\cap\zhat$ in an order preserving way (see\,\cite[Corollary 1.2.6]{bosch}. Therefore just as for $\zhat$  we have chains of prime ideals associated to ultrafilters of $\be\P$, the only difference is that if $\d$ is principal, say generated by $p$, the maximal ideal $\d^*$ of $\zhat$ extended to $\ad$ is  $\d^*\ad=\ad$ since $p$ is invertible in $\ad$ (see Corollary\,\ref{princ}), and so $\d_*\ad$ is both minimal and maximal in $\spec\ad$. In particular, we have that $\Min\ad=\{\d_*\ad: \d\in\be\P\}$, where $\d_*$ is the minimal prime ideal of $\zhat$ associated to $\d$, defined in  section\,\ref{prodvaldom}.
\begin{thm}

$(1)$  Any prime ideal of $\ad$ is contained in a unique maximal ideal.

$(2)$  Every maximal ideal of $\ad$ contains a unique minimal prime.

$(3)$ For any $\fq\in\spec\ad$, the set of ideals of $\ad$ containing $\fq$ is linearly ordered by inclusion.

\end{thm}
\begin{proof}
(1) and (2) by Corollary\,\ref{c1d*} and (3) by Theorem\,\ref{idealslo}.
\end{proof}
Let $\fq\in\spec\ad$, $\fp=\fq\cap\zhat\in\spec\zhat$ and $\d$ the ultrafilter associated to $\fp$.  

(a) If $\fp$ is principal we get $\ad/\fq\cong\Q_p$ for some $p\in\P$. If $\fp$ is nonprincipal, 
since $\ad/\fq\cong (\zhat/\fp)_{\Z^*}$, by Corollary\,\ref{invertible} $ (\zhat/\fp)_{\Z^*}\cong \zhat/\fp$. Therefore  Theorem\,\ref{quotientsch} and Theorem\,\ref{quotients} apply if we substitute $\zhat/\fp$ by  $\ad/\fq$. 

(b)  $\ad_\fq$ is isomorphic to $(\zhat_\fp)_{\Z^*}$, and since $\fp\cap\Z^*=\emptyset$ we get $\ad_\fq\cong\zhat_\fp$. Therefore,
  Corollary\,\ref{clocalization}  and Theorem\,\ref{pain}   hold for  $\ad_\fq$ too.

\bibliographystyle{plain}
\bibliography{biblio}

\def\Dbar{\leavevmode\lower.6ex\hbox to 0pt{\hskip-.23ex \accent"16\hss}D}
\begin{thebibliography}{10}

\bibitem{blassmilden}
Andreas Blass and Heike Mildenberger.
\newblock On the cofinality of ultrapowers.
\newblock {\em J. Symbolic Logic}, 64(2):727--736, 1999.

\bibitem{bosch}
Siegfried Bosch.
\newblock {\em Algebraic geometry and commutative algebra}.
\newblock Universitext. Springer, London, 2013.

\bibitem{canjarnotCH}
Michael Canjar.
\newblock Countable ultraproducts without {CH}.
\newblock {\em Ann. Pure Appl. Logic}, 37(1):1--79, 1988.

\bibitem{canjarcof}
R.~Michael Canjar.
\newblock Cofinalities of countable ultraproducts: the existence theorem.
\newblock {\em Notre Dame J. Formal Logic}, 30(4):539--542, 1989.

\bibitem{cassels}
J.~W.~S. Cassels.
\newblock Global fields.
\newblock In {\em Algebraic {N}umber {T}heory ({P}roc. {I}nstructional {C}onf.,
  {B}righton, 1965)}, pages 42--84. Thompson, Washington, D.C., 1967.

\bibitem{changkeisler}
C.~C. Chang and H.~J. Keisler.
\newblock {\em Model theory}, volume~73 of {\em Studies in Logic and the
  Foundations of Mathematics}.
\newblock North-Holland Publishing Co., Amsterdam, third edition, 1990.

\bibitem{comfort}
W.~W. Comfort and S.~Negrepontis.
\newblock {\em The theory of ultrafilters}.
\newblock Springer-Verlag, New York-Heidelberg, 1974.
\newblock Die Grundlehren der mathematischen Wissenschaften, Band 211.

\bibitem{MO71pmrings}
Giuseppe De~Marco and Adalberto Orsatti.
\newblock Commutative rings in which every prime ideal is contained in a unique
  maximal ideal.
\newblock {\em Proc. Amer. Math. Soc.}, 30:459--466, 1971.

\bibitem{englerprestel}
Antonio~J. Engler and Alexander Prestel.
\newblock {\em Valued fields}.
\newblock Springer Monographs in Mathematics. Springer-Verlag, Berlin, 2005.

\bibitem{fornasiero}
Antongiulio Fornasiero.
\newblock Embedding {H}enselian fields into power series.
\newblock {\em J. Algebra}, 304(1):112--156, 2006.

\bibitem{gouvea}
Fernando~Q. Gouv{\^e}a.
\newblock {\em {$p$}-adic numbers}.
\newblock Universitext. Springer-Verlag, Berlin, 1993.
\newblock An introduction.

\bibitem{Matsumura}
Hideyuki Matsumura.
\newblock {\em Commutative ring theory}, volume~8 of {\em Cambridge Studies in
  Advanced Mathematics}.
\newblock Cambridge University Press, Cambridge, second edition, 1989.
\newblock Translated from the Japanese by M. Reid.

\bibitem{mildenshelahcof}
Heike Mildenberger and Saharon Shelah.
\newblock The minimal cofinality of an ultrapower of {$\omega$} and the
  cofinality of the symmetric group can be larger than {$\mathfrak b^+$}.
\newblock {\em J. Symbolic Logic}, 76(4):1322--1340, 2011.

\bibitem{Prest}
Mike Prest.
\newblock {\em Model theory and modules}, volume 130 of {\em London
  Mathematical Society Lecture Note Series}.
\newblock Cambridge University Press, Cambridge, 1988.

\end{thebibliography}
\end{document}